\documentclass[a4paper, reqno]{amsart}

\usepackage{amsthm}
\usepackage{amsfonts}
\usepackage{amsmath}
\usepackage{mathrsfs}
\usepackage[pdfauthor={Oliver Butterley},
		pdftitle={Expanding Semiflows on Branched Surfaces and   One-Parameter Semigroups of Operators},
		pdfsubject={Semigroups and Expanding Flows},
		pdfkeywords={Branched manifold, Expanding flow, Operator semigroups},
		colorlinks=true,
		citecolor=black,
		linkcolor=black]{hyperref}
\usepackage{graphicx}
\usepackage{epstopdf}
\usepackage{bbm}
\usepackage{MnSymbol}
\theoremstyle{theorem}
\newtheorem{thm}{Theorem}[section]
\newtheorem{thmm}{Theorem}
\newtheorem{lem}[thm]{Lemma}

\newtheorem{prop}[thm]{Proposition}

\newtheorem{defin}[thm]{Definition}
\theoremstyle{definition}

\newtheorem{rem}[thm]{Remark}

\DeclareMathAlphabet{\mathsc}{OT1}{cmr}{m}{sc}

\newcommand\cB{{\mathscr B}}
\newcommand\cC{{\mathscr C}}

\newcommand\cG{{\mathscr G}}

\newcommand\cL{{\mathscr L}}

\newcounter{saveenum}
\newcommand\bA{{\mathbb A}}

\newcommand\bC{{\mathbb C}}

\newcommand\bN{{\mathbb N}}

\newcommand\bR{{\mathbb R}}

\newcommand\B{{\mathfrak B}}

\newcommand{\measures}{\mathfrak{M}}
\newcommand{\dmeasures}{ \mathfrak{D}}
\newcommand{\Oo}{\Omega} 
\newcommand{\U}{\mathbf{U}} 
\newcommand{\Ichart}{\mathcal{I}} 
\newcommand{\Isubchart}{\mathcal{J}} 
\newcommand{\V}[2]{\mathbf{V}_{{#1}{#2}}} 
\newcommand{\var}{\varphi} 
\newcommand{\al}{\alpha} 
\newcommand{\Int}[1]{\operatorname{Int} #1} 
\newcommand{\D}{D} 
\newcommand{\flo}[1]{\Phi^{{#1}}} 
\newcommand\cF{{\mathscr F}} 
\newcommand{\m}{\operatorname{\mathbf{m}}} 
\newcommand{\Id}{\mathbf{id}}   
   
\newcommand{\TotalVar}[1]{\norm{\smash{ #1 }}_{\measures}}   

\newcommand{\supnorm}[1]{{\lvert{#1}\rvert}_{\infty}}

\newcommand{\norm}[1]{\left\lVert{#1}\right\rVert}
\newcommand{\abs}[1]{\left\lvert{#1}\right\rvert}

\newcommand{\R}[1]{\mathcal{P}(#1)}

\begin{document}
\author{Oliver Butterley}
\email{oliver.butterley@gmail.com}
\address{Oliver Butterley\\
Dipartimento di Matematica\\
II Universit\`{a} di Roma (Tor Vergata)\\
Via della Ricerca Scientifica, 00133 Roma, Italy.}
\keywords{expanding flow, transfer operator, branched manifold, spectral gap, one-parameter semigroup}
\title[Expanding Semiflows on Branched Surfaces]{Expanding Semiflows on Branched Surfaces and   One-Parameter Semigroups of Operators}
\thanks{It is a pleasure to thank Carlangelo Liverani for many helpful discussions and comments. Research partially supported by the ERC Advanced Grant MALADY (246953). I am indebted  to Stefano Luzzatto for invaluable assistance during a period of many years.  I am grateful  to the library at ICTP where much of this work was done.}

\maketitle

\thispagestyle{empty}

\begin{abstract}
We consider expanding semiflows on  branched surfaces. The family of transfer operators associated to the semiflow is a one-parameter semigroup of operators. The  transfer operators may also be viewed as an operator-valued function of time and so, in the appropriate norm, we may consider the vector-valued Laplace transform of this function. We obtain a spectral result on these operators and relate this to the spectrum of the generator of this semigroup. Issues of strong continuity of the semigroup are avoided. The main result is the improvement to the machinery associated with studying semiflows as one-parameter semigroups of operators and the study of the smoothness properties of semiflows defined on branched manifolds, without encoding as a suspension semiflow.
\end{abstract}

\section{Introduction}
Flows associated to vector fields were one of the principle origins of the study of ergodic theory and dynamical systems and are indeed of foremost importance. Frequently they are not at all simple to analyse.  Certain deceptively simple systems of differential equations and the associated flows  still prove extremely difficult to understand.
In the past many questions concerning flows were intractable with the technology available and much progress was made by first reducing to a discrete dynamical systems by considering Poincar\'e sections and encoding the flow as a suspension over the discrete-time dynamical system.  

In the study of the statistical properties of discrete-time dynamical systems a major technological success of the last thirty years was the development of ideas to apply functional analysis directly to the system. This was developed by a long list of people but particularly by the pioneering work of Lasota-Yorke \cite{lasota1981led}  and subsequent development by Keller (see \cite{LiView} and references within for a more complete history). In this approach one typically considers a linear operator called the ``transfer operator'' acting on  a certain well chosen Banach space and then deduces information concerning the statistical properties from information concerning the spectrum of the operator.

When Liverani studied the rate of mixing of  contact Anosov flows  \cite{Li1} he showed that the family (parametrized by time) of transfer operators associated to a flow can be viewed as a strongly-continuous one-parameter semigroup acting on a well chosen Banach space. This had the benefit of allowing one to study  the flow directly without first encoding to a suspension flow and again apply the breakthrough work of Dolgophyat  \cite{D} on the oscillatory cancelation mechanism. This seemed like a point of view which had great potential and indeed these ideas have since been proven useful. In particular they  have  helped deduce behaviour of the invariant measure  of an Anosov flow  under perturbations \cite{BL}, to study the rate of mixing for piecewise  cone-hyperbolic contact flows \cite{BaLiPwCone} and to study dynamical zeta functions, again for contact Anosov flows \cite{Giulietti:fk}. We remark that although studying the flow by considering the associated one-parameter semigroup of operators seems promising it is not the only possibility and Tsujii has demonstrated  \cite{Ts,tsujii2008qct,Tsujii:2011uq} that, certainly in the case of smooth expanding maps of the circle and contact Anosov flows, it is possible to study directly the transfer operator associated to the time-one map of the flow.

Many statistical properties of many diverse classes of flows remain as open questions, for example rates of mixing for the Sinai billiard flow and the Lorenz flow. These are both flows which are simple to define but whose statistical properties remain elusive (the corresponding questions for the Poincar\'e return maps associated to these flows are relatively well understood).  From a technological point of view several issues must be better understood if we wish to extend our techniques to more general classes of flows, in particular understanding how to deal with discontinuities. 
A weight of evidence suggests that ``good statistical properties''  like exponential decay of correlation and continuous, or even differential, dependence of the invariant measure under perturbations, are the consequence of the smoothness of the system.  The aim here is to use as much as possible the available smoothness of the system to deduce statistical properties in the situations where there is a limited degree of smoothness. In particular this is why we avoid the approach of reducing to a suspension flow which  artificially reduces the smoothness of the system.

We believe there are many benefits to streamlining and optimising the current technology to facilitate its use in more difficult settings. 
 As mentioned above, there is now a very precise understanding  of $\cC^{r}$ Anosov flows and extremely precise spectral results, however, from a physical point of view such  boundary-less smooth systems seem unrealistic. Here we wish to consider the more realistic systems which only satisfy  significantly weaker regularity assumptions.
As such we study a relatively simple model, although much subtle and complex behaviour is visible and the results are indeed new.
We study semiflows associated to $\cC^{2}$ vector fields on two-dimensional branched manifolds (branched surfaces), possibly with boundary. 
Despite the smoothness of the flow  discontinuities are introduced because the flow is supported on a manifold with boundary.   
The existence of branches allows the semiflows to be non-invertible, i.e. they really are semiflows and not flows. We suppose these semiflows are uniformly expanding in a sense made precise below.
We develop the theory of the one-parameter semigroup of transfer operators associated with these semiflows and we make several improvements and observations from a technical point of view. We achieve a spectral decomposition of the (operator-valued) Laplace transform of the transfer operator. We show that the issue of the strong continuity of the semigroup can be easily avoided. Furthermore we demonstrate an approach which means that it should be possible to study also perturbations of the flow on the same Banach space, even in the case of flows with discontinuities, and so studying the behaviour of the statistical properties under the perturbation is made possible.
We remark that the operator-theory framework presented in the following section is essentially independent of the present application to expanding flows on branched manifolds and should be applicable, with the appropriate choice of dynamically relevant Banach space, to many other settings.

\section{Results}

We suppose that $\Oo$ is a 2-dimensional $\cC^{2}$ branched manifold, possibly with boundary and with finite branches. Definitions and notation concerning branched manifolds and their differential structure are given in Section~\ref{sec:branchedmanifolds}. In summary a branched manifold possesses a differential structure in much the same way as a Riemannian manifold, in particular tangent space $T_{p}\Oo$ is uniquely defined at each point $p\in \Oo$ and there is an inner product $\langle \cdot, \cdot \rangle$ for the tangent space which allows us to discuss orthogonality and consequently a norm $\norm{\cdot}$.
We  suppose  that we are given a vector field $X\in T\Oo$  which is $\cC^{2}$ and  such that the associated semiflow
\[
\flo{t}: \Oo \to \Oo, \quad \quad t\geq 0
\]
is globally defined. By \emph{semiflow} we mean, as usual, that  $\Phi : \Oo \times \bR_{{+}} \to \Oo$ which we write as  $\Phi : (p,t) \mapsto \flo{t}(p)$ and which satisfies $\flo{0} = \Id$ and $\flo{t} \circ \flo{s} = \flo{t+s}$ for all $t,s \geq 0$. 

We also require that $ \flo{t} $ is uniformly expanding as made precise in the following.
To characterise hyperbolicity for systems which are either not invertible or have discontinuities it is not possible to use the notion of an invariant and  hyperbolic splitting of tangent space. One possibility would be to use the notion of conefields. However we opt for yet another alternative which is most suitable for this particular setting.
We suppose there exists an orientatable $\cC^{2}$ foliation of $\Oo$ which we denote $\cF_{V}$ such that the following three properties hold.
\begin{enumerate}
\item
The leaves of $\cF_{V}$ are all one-dimensional curves with end points contained in $\partial \Oo$ and of length greater than $\delta$ for some constant $\delta>0$,
\item
The leaves of $\cF_{V}$ are uniformly transversal to the flow direction,
 \setcounter{saveenum}{\value{enumi}}
\end{enumerate}
and, letting $V$ denote the unit vector field tangent to the foliation $\cF_{V}$,  we suppose that there exists constants  $C<\infty$, $\lambda>0$ such that
\begin{enumerate}
  \setcounter{enumi}{\value{saveenum}}
\item
$\langle  D\flo{t} u , V \rangle \geq C^{-1} e^{\lambda t} \langle u , V \rangle$ for all $u\in T_{p}\Oo$, $p\in \Oo$ and $t\geq0$.
\end{enumerate}
Note well that the foliation $\cF_{V}$ will not be invariant under the action of the flow, except in extremely special (non-mixing) cases. The foliation in merely more-or-less in the expanding direction of the flow. Also note that the assumption of the existence of the flow and of it being uniformly expanding in the above sense puts significant restrictions on the branched manifold. For example the flow lines at the boundary must be tangent to the boundary. Many branched manifolds cannot support such flows.

%
%
%

From this point onwards we assume always that the semiflow $\flo{t} : \Oo \to \Oo$ is  a uniformly expanding semiflow on a two dimensional branched manifold as described above.

\begin{figure}[tbp]
\begin{center}
\includegraphics[scale=0.55]{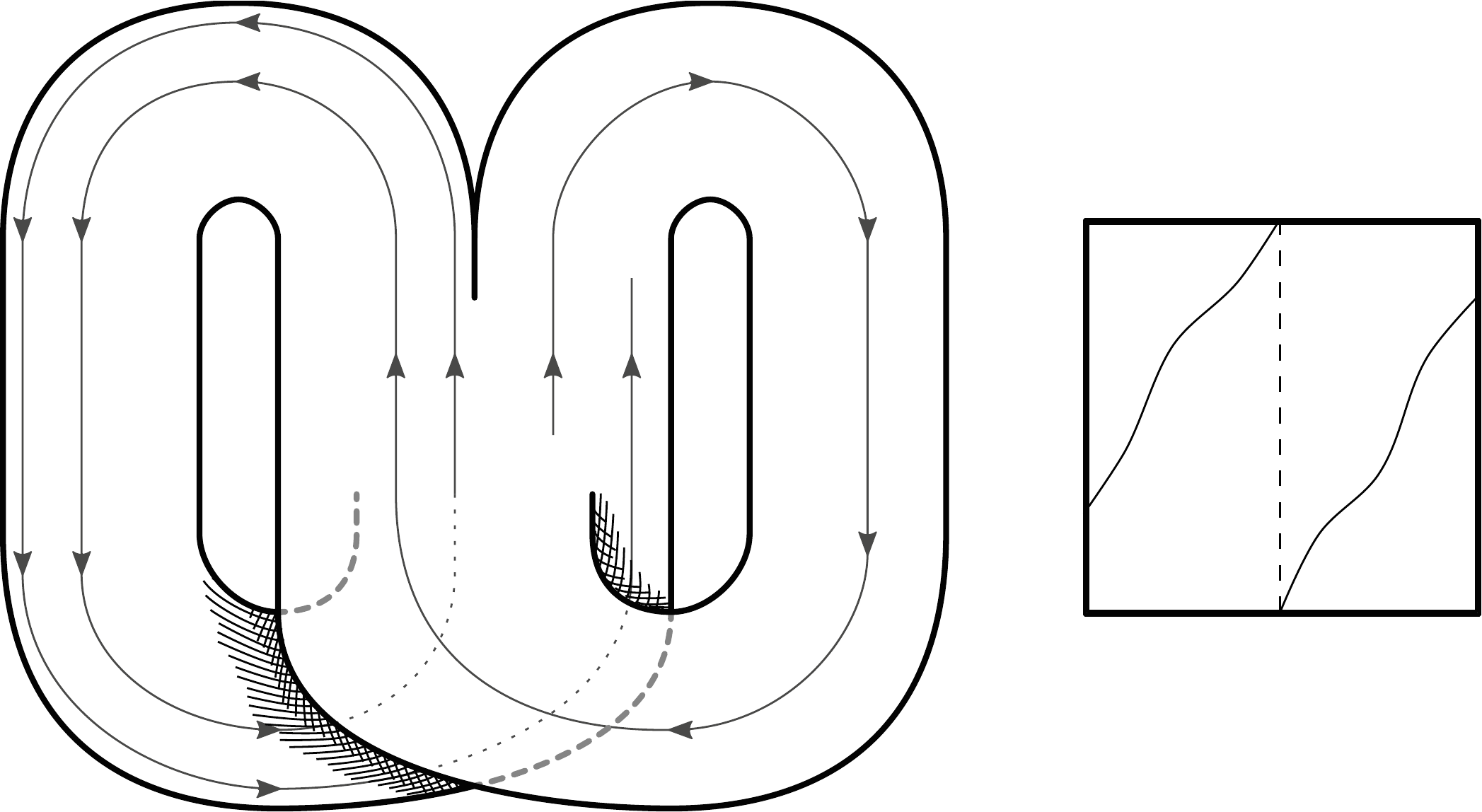}
\caption{{The prototype of a branched manifold which supports an expanding semiflow. A piece of an orbit of the flow is displayed (the line with the arrows).  The semiflow fails to be invertible. 
There is a single branch line, horizontally across the central part where the left and right surfaces join.
To the right is the one-dimensional Poincar\'e return map produced by taking the horizontal line across the middle as the Poincar\'e section.}}
\label{fig:branchedmanifold}
\end{center}
\end{figure}

The branched manifold $\Oo$ is a measure space when equipped with the Borel $\sigma$-algebra.  Let $\measures$ denote the space of complex measures on $\Oo$.  This is the dual of $\cC_{0}(\Oo)$, the Banach space of continuous complex-valued functions with  support contained within some open subset of $\Oo$, equipped  with the supremum norm  $\supnorm{\eta}:= \sup\{ \abs{\eta(p)}: p\in \Oo\}$. 
For each  $\mu \in \measures$  let $\TotalVar{\mu}=\sup\{ \abs{ \mu(\eta)}: \eta\in \cC(\Oo), \supnorm{\eta}\leq 1\}$. The space $(\measures, \TotalVar{\cdot})$ is a Banach space. Note that $\TotalVar{\cdot}$ is exactly the standard total variation which is historically denoted by $\abs{\cdot}\!(\Oo)$ but for clarity and consistency in the following we use the norm style notation.
We will refer to the elements of this space as measures  and omit  explicit mention that they are complex measures. 
For later use let $\cC(\Oo)$ (as opposed to $\cC_{0}(\Oo)$) denote the space of all continuous complex-valued functions on $\Oo$.
For each fixed $t\geq 0$ the flow $\flo{t}:\Oo \to \Oo$ is a measurable map and so defines the push-forward in the space of measures.
\[
\flo{t}_{*} : \measures \to \measures, \quad \flo{t}_{*}\mu(\eta) := \mu(\eta \circ \flo{t})
\]
for all  $\eta \in \cC(\Oo)$. This family of linear operators is a one-parameter semigroup since $\flo{0}_{*} =\Id$ and it has the semigroup property, inherited  from the semiflow, that
\begin{equation}\label{eq:semigroupproperty}
\flo{t+s}_{*} = \flo{t}_{*} \circ \flo{s}_{*} \quad \quad \text{for all $t,s\geq 0$}.
\end{equation}
At this stage we make no claims on the continuity of this semigroup with respect to the parameter $t$. This is however an important issue that we return to later on.
They are  a family of bounded linear operators $\flo{t}_{*} : \measures \to \measures$.
We use the standard notation for the operator norm. I.e. if $\cL:\measures \to \measures$ is a linear operator then $\TotalVar{\cL} := \sup\{ \TotalVar{\cL\mu} : \TotalVar{\mu}\leq 1\}$.
\begin{lem}\label{lem:TotalVarBound}
$\TotalVar{\flo{t}_{*}} =1$ for all $t\geq 0$.
\end{lem}
\begin{proof}
For any $\mu \in \measures$ then $\TotalVar{\flo{t}_{*}\mu} = \sup\{  \abs{\mu(\eta\circ \flo{t})}: \eta\in \cC(\Oo), \supnorm{\eta}\leq 1\}$. 
That $\abs{\eta }_{\cC^{0}}\leq 1$  implies that $\abs{\eta \circ \flo{t} (p)}\leq 1$ for each $p\in \Oo$  and also $\eta \circ \flo{t} $ is measurable and so by Lusin's Theorem   $\TotalVar{\flo{t}_{*}\mu} \leq \TotalVar{\mu}$. We have shown that $\TotalVar{\flo{t}_{*}} \leq1$ for all $t\geq 0$. Define the linear functional $\ell \in \measures^{*}$ by $\ell(\mu):= \mu(1)$ for all $\mu\in \measures$.  Note that $\ell(\flo{t}_{*}\mu):= \flo{t}_{*}\mu(1) = \mu(1)$ and so $1$ is an eigenvalue for the dual operator and consequently $1$ is in the spectrum of $\flo{t}_{*}:\measures \to \measures$.
\end{proof}

For flows it tends to be difficult to study the operator $\flo{t}_{*}$ directly and we introduce a related family of operators in the following. First some notation: For any pair of Banach spaces $\mathfrak{A}$, $\B$ we use the notation $ \cB(\mathfrak{A}, \B)$ to denote the space of bounded linear operators from $\mathfrak{A}$ to $\B$.
By Lemma~\ref{lem:TotalVarBound} we know that $\int_{0}^{\infty} e^{-\Re(z)t} \TotalVar{\flo{t}_{*}} \ dt  <\infty$ for all  $\Re(z)>0$ and so the function $t\mapsto e^{-zt} \flo{t}_{*} \in \cB(\measures, \measures)$ is Bochner integrable \cite[\S V.5]{MR1336382}. We  define the operator $\R{z}: \measures \to \measures$ by
\begin{equation}\label{eq:defRz}
\R{z}  := \int_{0}^{\infty} e^{-zt} \flo{t}_{*}  \ dt, \quad \text{for all $\Re(z)>0$}.
\end{equation}
In the following we see that $\R{z}$ is  a pseudo-resolvent, a consequence of  $ t\mapsto \flo{t}_{*} $ having the semigroup property \eqref{eq:semigroupproperty}.
\begin{lem}\label{lem:resolventequation}
For all $\Re(z)>0$, $\Re(\zeta)>0$ then
$(z-\zeta) \R{\zeta} \R{z} = \R{\zeta} - \R{z}$.
\end{lem}
\begin{proof}
Without loss of generality  we assume that $\Re(\zeta - z)>0$ and $\zeta \neq z$.
By definition~\eqref{eq:defRz} for all $\mu\in \measures$
\[
\R{\zeta} \R{z} \mu = \int_{0}^{\infty} \int_{0}^{\infty} e^{-\zeta s} e^{-z t} \flo{t+s}_{*} \mu \ dt \ ds.
\]
Changing variables $u=s+t$, splitting the integral into two pieces and then swapping the two integral in the second piece we have
\[\begin{split}
\R{\zeta} \R{z} \mu &= \int_{0}^{\infty} \int_{s}^{\infty}e^{-(\zeta-z) s} e^{- z u} \flo{u}_{*} \mu \ du \ ds\\
 &= \int_{0}^{\infty}  e^{-(\zeta-z) s}  \ ds  \  \R{z} \mu 
 -  \int_{0}^{\infty} \int_{0}^{s}e^{-(\zeta-z) s} e^{- z u} \flo{u}_{*} \mu \ du \ ds\\
 &= \int_{0}^{\infty}  e^{-(\zeta-z) s}  \ ds  \  \R{z} \mu 
 -  \int_{0}^{\infty} \left( \int_{u}^{\infty}e^{-(\zeta-z) s}  \  ds  \right) e^{- z u} \flo{u}_{*} \mu \ du .
\end{split}\]
Since $ \int_{0}^{\infty}  e^{-(\zeta-z) s}  \ ds  = (\zeta - z)^{-1}$ and $  \int_{u}^{\infty}e^{-(\zeta-z) s}  \  ds = (\zeta -z)^{-1}  e^{-(\zeta-z) u} $ then the above calculation implies that $\R{\zeta} \R{z} \mu = (\zeta - z)^{-1} \left(  \R{\zeta} -  \R{z} \right)$ as required.
\end{proof}

\begin{lem}\label{lem:boundRz1}
$\TotalVar{\R{z}} \leq \Re(z)^{-1}$ for all $\Re(z)>0$. 
\end{lem}
\begin{proof}
Note that $\TotalVar{\R{z} } \leq \int_{0}^{\infty} e^{-\Re(z)t} \TotalVar{\flo{t}_{*} } \ dt$ by \eqref{eq:defRz}.
We use the bound on $\TotalVar{ \flo{t}_{*} }$ from Lemma~\ref{lem:TotalVarBound} and integrate $\int_{0}^{\infty} e^{-\Re(z)t}  \ dt = \Re(z)^{-1}$.
\end{proof}

\begin{lem}\label{lem:Rzn}
For each  $n\in \bN$ and $z\in \bC$ with $\Re(z)>0$ 
\begin{equation*}
\R{z}^{n}  =\frac{ 1 }{(n-1)!}  \int_{0}^{\infty} t^{n-1}e^{-zt} \flo{t}_{*}  \ dt.
\end{equation*}
\end{lem}
\begin{proof}
This is a direct consequence of the definition \eqref{eq:defRz} by induction on $n$ changing variables in the double integral produced and then swapping the order of integration.
\end{proof}

The natural reference measure  is   the $2$-dimensional Hausdorff measure on $\Oo$ which we denote by $\m\in \measures$. If one decided to work with densities using the charts of the differential structure of $\Oo$ one could equivalently consider Lebesgue measure as the reference. 
We do not expect the operators $\flo{t}_{*}$ and $\R{z}$ to have good spectral properties acting on $\measures$ and we are only interested in properties which are ``physically relevant'' in the sense of relating to  measures which are absolutely continuous with respect to $\m$.
Therefore we may and it is beneficial to consider a Banach space of measures which is contained in $\measures$  and  on which $\R{z}$ acts with good spectral properties.

Let $U$ be a continuous vector field. A  measure $\mu \in \measures$ is said to be differentiable (in the sense of measures) with respect to $U$
if there exists a measure $\D_{U}\mu \in \mathfrak{M}$ such that $\mu(U\eta) = - \D_{U}\mu(\eta)$ for all $\eta \in \cC^{1}_{0}(\Oo)$. The differential of a measure is the linear map 
$\D\mu : T\Oo \times \cC(\Oo) \to \bC$, $\D\mu : (U, \eta) \mapsto \D_{U}\mu(\eta)$. 
The tangent space at each point is of course finite dimensional and so there are many equivalent possibilities for the definition of the norm.
For our purposes it is convenient to have a coordinate-independent definition of the norm.
Let $\mathfrak{X}(\Oo)$ denote the set of continuous vector fields on $\Oo$. With a slight abuse of notation let
\[
\TotalVar{\D \mu} := \sup\{  \TotalVar{\D_{U}\mu } : U\in \mathfrak{X}(\Oo), \supnorm{U} \leq 1\}.
\]
If this quantity is finite we say that  $\mu \in \measures$ is differentiable (in the sense of measures) and we let $\mathfrak{D}(\Oo)$ denote the set of all such measures. For all $\mu \in\dmeasures$ let
\[
\norm{ \mu }_{\dmeasures} := \TotalVar{\D\mu} + \TotalVar{\mu}.
\]
If we were to consider the densities of the measures this norm is nothing other that the   bounded variation norm. 
However, due to the oddities of working on a branched manifold, it is most convenient to work directly with the measures, seeing them as linear functionals, as apposed to working with the corresponding densites. 
The  Banach space $(\dmeasures, \norm{\cdot}_{\dmeasures})$ is the central component of this study.

\begin{prop}\label{prop:bounded}
There exists $C<\infty$, $\alpha<\infty$ such that $\norm{\smash{\flo{t}_{*}}}_{\dmeasures} \leq Ce^{\alpha t}$ for all $t\geq 0$.
\end{prop}
\begin{flushleft}The proof of the above proposition is the content of  Section~\ref{sec:bounded}.
\end{flushleft} 

%
%
  The above estimates are far from optimal and later we will be able to improve them but they are required for us to proceed at this stage.
A first consequence of the above lemma is that we may also consider $\R{z} : \dmeasures \to \dmeasures$ for all $\Re(z)>\alpha$.
The space $\cB(\dmeasures, \measures)$, which is the space of  linear operators  $\cL: \dmeasures \to \measures$  is, when endowed with the operator norm 
\[
\norm{ \cL  }_{\dmeasures \to \measures}:= \sup\{ \norm{\cL \mu}_{\measures}: \mu \in \dmeasures, \norm{\mu}_{\dmeasures} \leq 1  \},
\]
   a Banach space. It is interesting to note that this idea of using a weaker operator norm, considering the operator as mapping from strong space to weak space, has already been used to great effect in studying the stability of the spectrum of discrete-time dynamical systems \cite{KL}.
\begin{lem}\label{lem:lip}
There exists $C<\infty$ such that $\frac{1}{t} \norm{ \smash{\flo{t}_{*} - \Id } }_{\dmeasures \to \measures} \leq C$ for all $t> 0$.
\end{lem}
\begin{proof} 
Since $ \frac{1}{t} \norm{ \smash{\flo{t}_{*} - \Id } }_{\dmeasures \to \measures} \leq \frac{1}{t_{0}} \norm{ \smash{\flo{t}_{*} - \Id } }_{\measures} \leq 2/t_{0}  $ for all $t\geq t_{0}$ it suffices to prove the lemma for  $t\in  [0,t_{0}]$ where $t_{0}$ is small. We estimate $\sup\{\abs{(  \flo{t}_{s}\mu - \mu  )(\eta)}: \eta\in \cC^{1}(\Oo), \supnorm{\eta}\leq 1\}$. Fix $\mu \in \dmeasures$.
The key is to note  that 
$\int_{0}^{t} X \eta \circ \flo{s} \ ds = \eta \circ \flo{t} - \eta$ for all $t\geq 0$.
This means that for any $\eta\in \cC^{1}(\Oo)$
\[
\left(  \flo{t}_{s}\mu - \mu  \right)(\eta) = \mu( \eta\circ\flo{t} - \eta)
= \int_{0}^{t} \flo{s}_{*}\mu(X\eta) \ ds.
\]
Using the estimate from Proposition~\ref{prop:bounded} this shows that there exists $C<\infty$ such that $\TotalVar{ \flo{t}_{s}\mu - \mu } \leq t C \norm{\mu}_{\dmeasures}$ for all $\mu\in \dmeasures$.
\end{proof}
Using also the semigroup property the above lemma says that the operator-valued function $t\mapsto \flo{t}_{*} \in \cB(\dmeasures, \measures)$ is Lipschitz. I.e. is Lipschitz with respect to the $\norm{ \cdot  }_{\dmeasures \to \measures}$ norm. 
We can now make clear one of the reasons why the quantity $\R{z}$ is so important, namely that the key behaviour of $\flo{t}_{*}$ maybe be recovered from the study of $\R{z}$. Here we take the point of view that the definition of $\R{z}$ is akin to the Laplace-Stieltjes transform of an operator valued function and so, in a limited sense, there exists an inverse to this transform.
\begin{thmm}\label{thm:inverse}
Suppose $t\geq 0$, $a >0$. Then, in $\cB(\dmeasures, \measures)$, we have that
\[
  \flo{t}_{*}   = \lim_{k\to \infty} \frac{1}{2\pi i} \int_{-k}^{k} e^{(a+ib)t}\R{a+ib} \ db.
\]
\end{thmm}
\begin{proof}
This is  an application of the inverse of the Laplace-Stieltjes transform of an operator valued function so we provide the details to pass from our setting here to the result described in the reference monograph \cite{arendt2011vector}.
 Let $F(t):= \flo{t}_{*} - \Id$. Since $F(0)=0$ by definition and by Lemma~\ref{lem:lip} we have that $F \in \operatorname{Lip_{0}}(\bR_{+},\cB(\dmeasures, \measures)   )$ (as defined in the reference). Consequently, by \cite[Theorem~2.3.4 and (1.22)]{arendt2011vector}, we have that
 \[
 F(t) = \lim_{k\to \infty} \frac{1}{2\pi i} \int_{-k}^{k} e^{(a+ib)t} \left( R(a+ib) - \frac{\Id}{a+ib}    \right) \ db.
 \]
This is sufficient to conclude since $\int_{-\infty}^{\infty} (a+ib)^{-1} e^{(a+ib)t} \ db = 2 \pi i$.
\end{proof}
The above lemma could not be expected to hold in $\cB(\dmeasures, \dmeasures)$ since the Lipschitz property of Lemma~\ref{lem:lip} is essential. We note in passing that other possibilities  exist, including considering the inverse for some fixed measure \cite[\S I.3.12]{arendt2011vector}.
To proceed we clearly require more information concerning  $\R{z}$.
\begin{prop}\label{prop:quasicompact}
For each $\Re(z)>\alpha$ the linear operator $\R{z}:\dmeasures \to \dmeasures$ is quasi-compact with spectral radius bounded above by $\Re(z)^{-1}$  and essential spectral radius bounded above by $(\Re(z)+\lambda)^{-1}$.
\end{prop}

\begin{flushleft}The proof of the above proposition is the content of Section~\ref{sec:proofQC}.\end{flushleft}

   The  importance of the above proposition lies in the following consequence, a fact which has significant relevance in view of Theorem~\ref{thm:inverse}.

\begin{thmm}\label{thm:extension}
The operator valued function $z\mapsto \R{z} \in \cB(\dmeasures, \dmeasures)$  admits an extension which is holomorphic on the set $\{z\in \bC: \Re(z)>0\}$ and  meromorphic   on the set $\{z\in \bC: \Re(z)>-\lambda\}$.
\end{thmm}
\begin{proof}
Fix $z\in \bC$ such that $\Re(z)>\alpha$. For any $\eta\in \bC$ such that $\abs{\eta} > \Re(z)^{-1}$ then  
$ \eta^{-1} \R{z+\eta^{-1}} \R{z} = \R{z+\eta^{-1}} - \R{z}$
   by Lemma~\ref{lem:resolventequation} since in particular  $\eta \neq 0$ and $\Re(z-\frac{1}{\eta})>0$. Rearranging we obtain
\[
\R{z+\tfrac{1}{\eta}} = \eta \R{z} (\eta \Id - \R{z})^{-1}
\]
using that $(\eta \Id - \R{z})$ is invertible by the spectral radius estimate of  Proposition~\ref{prop:quasicompact}.
We use this formula to define the extension of $\R{z}$ into the left half of the imaginary plane.
 By the essential spectral radius estimate of Proposition~\ref{prop:quasicompact} the operator valued function $\eta \mapsto (\eta \Id - \R{z})^{-1}$ is meromorphic on the set $\{ \abs{\eta} > (\Re(z) +\lambda)^{-1} \}$.
\end{proof}

We now relate the above ideas to the standard theory of one-parameter semigroups  (see   \cite{Davies:2007qy} for the theory of one-parameter semigroups).
The \emph{generator} of the one-parameter semigroup $\flo{t}_{*}$ is the linear operator 
defined by
\[ 
Z\mu := \lim_{t\to 0} \tfrac{1}{t}  \left(\flo{t}_{*} \mu -\mu   \right)
\]
the domain of $Z$ being the set of $\mu\in \dmeasures$ for which the limit exists.
We would expect $Z$ to be an unbounded operator and moreover,  there is no reason to expect even that the domain of $Z$ is dense in $\dmeasures$. The problem is that there is no reason for the range of the resolvent to be dense in the setting that we are studying. However is the semigroup of operators were \emph{strongly continuous}\footnote{Strongly-continuous one-parameter semigroups are sometimes called $c_{0}$-semigroups.} then by standard theory the domain of the generator is dense.  If we wished to consider the generator in this fashion we may always take the following approach. For all $s>0$, $\mu \in \measures$ let
\[
\bA_{s}\mu := s^{-1} \int_{0}^{s} \flo{t}_{*}\mu \ dt.
\]
Hence let 
$\dmeasures_{\flat}:= \{ \bA_{s}\mu : \mu\in \dmeasures, s>0\}$ and let $\dmeasures_{\ddagger}$ denote the completion of $\dmeasures_{\flat}$ with respect to $\norm{\cdot}_{\dmeasures}$.
%
%
Since $\dmeasures_{\ddagger}$ is a vector subspace of $\dmeasures$ and complete by construction we know that $(\dmeasures_{\ddagger}, \norm{\cdot}_{\dmeasures})$ is a Banach space. Note also that $\flo{t}_{*} \dmeasures_{\ddagger} \subseteq \dmeasures_{\ddagger}$. 

\begin{lem}\label{lem:c0}
$\norm{\smash{\flo{t}_{*}\mu - \mu}}_{\dmeasures} \to 0$ as $t\to 0$ for all $\mu \in \dmeasures_{\ddagger}$.
\end{lem}
\begin{proof}
By density it suffices to prove the lemma for  $\nu =\bA_{s}\mu   $ where $\mu \in \dmeasures$, $s>0$. We have that 
\[
\flo{t}_{*}\nu - \nu 
= s^{-1}\left( \int_{0}^{t}  \flo{u+s}_{*} \mu \ du - \int_{0}^{t} \flo{u}_{*} \mu \ du \right).
\]
We conclude since $\flo{t}_{*}:\dmeasures \to \dmeasures$ is bounded as demonstrated in Proposition~\ref{prop:bounded}.
\end{proof}

The above lemma means that $\flo{t}_{*}: \dmeasures_{\ddagger} \to\dmeasures_{\ddagger} $ is a \emph{strongly-continuous} one-parameter semigroup. Therefore, 
by standard theory  \cite{Davies:2007qy} the domain of \( Z \) is a dense linear subspace of \( \dmeasures_{\ddagger}  \), moreover $Z$ is a closed operator and $\R{z} = (z\Id -Z)^{-1} = R(z,Z)$ for all $\Re(z) >0$. Note that what we are saying is that the resolvent operator of the generator coincides with the operator $\R{z}$ previously defined  by the integral \eqref{eq:defRz}.  Note that it is known 
 \cite[Problem 8.1.6]{Davies:2007qy}
that the range of a pseudo-resolvent, for example $\R{z}$, is independent of $z$ for all $z$ in the domain of definition. We use the notation $\operatorname{Ran}(\cdot)$ to denote the range of some linear operator. The space $ \dmeasures_{\ddagger}  $ is sufficiently large in the following sense.  In the statement of the following lemma we  mean the range of the operator $\R{z}:\dmeasures \to \dmeasures$ and not on some other domain.
\begin{lem}\label{lem:bigenough}
Suppose $\Re(z)>0$. Then $\operatorname{Ran}(\R{z}) \subseteq \dmeasures_{\ddagger}$.
\end{lem}
\begin{proof}
Fix $\Re(z)>0$ and let $\nu\in \operatorname{Ran}(\R{z})$. I.e. $\nu = \int_{0}^{\infty}e^{-zt} \flo{t}_{*}\mu \ dt$ for some $\mu \in \dmeasures$. 
This means that
\[
\bA_{s}\nu = s^{-1}\int_{0}^{s}\int_{0}^{\infty} e^{-zt} \flo{t+u}_{*}\mu \ dt \ du
= s^{-1}\int_{0}^{s} e^{zu}\int_{u}^{\infty} e^{-zw} \flo{w}_{*}\mu \ dw \ du.
\]
Consequently
\[
\bA_{s}\nu - \nu
= s^{-1}\int_{0}^{s} (e^{zu}-1) \ du \  \R{z}\mu 
 -  s^{-1}\int_{0}^{s} e^{zu}\int_{0}^{u} e^{-zw} \flo{w}_{*}\mu \ dw \ du.
\]
Since we know that $\flo{t}_{*}: \dmeasures \to \dmeasures$ is bounded the above calculation means that
\[
\norm{ \bA_{s}\nu - \nu}_{\dmeasures} \to 0, \quad\text{as $s\to 0$}.
\]
Since $ \bA_{s}\nu \in \dmeasures_{\flat}$ and $\dmeasures_{\ddagger}$ is defined as the completion of $ \dmeasures_{\flat}$ we have shown that $\nu \in \dmeasures_{\ddagger}$.
\end{proof}

We use the notation $\operatorname{Dom}(\cdot)$ to denote the domain of some linear operator.
In is convenient to introduce yet one more norm.
For all  $\mu \in \operatorname{Dom}(Z)$ let
\[
\norm{\mu}_{Z} := \norm{Z\mu}_{\dmeasures} +  \norm{\mu}_{\dmeasures}. 
\]
It is known \cite[Lemma~6.1.15]{Davies:2007qy} that $ \operatorname{Dom}(Z)$ is complete with respect to the above defined norm and so one could choose to consider the Banach space $(\operatorname{Dom}(Z), \norm{\cdot}_{Z})$. 
Moreover $\flo{t}:\operatorname{Dom}(Z) \to \operatorname{Dom}(Z)$ is a one-parameter semigroup.
An inspection of the proof of Lemma~\ref{lem:c0} shows that the operator-valued function\footnote{Here we follow the previous convention of notation and so $ \cB( \operatorname{Dom}(Z), \dmeasures)$ denotes the space of bounded linear operators $\cL: \operatorname{Dom}(Z) \to  \dmeasures $.} $t\mapsto \flo{t}_{*} \in \cB( \operatorname{Dom}(Z), \dmeasures)$ is Lipschitz since $\dmeasures_{\flat}\subset \operatorname{Dom}(Z)$ and is actually a core\footnote{See     \cite[Theorem~6.1.18]{Davies:2007qy}} for $Z$. Consequently the analogue of Theorem~\ref{thm:inverse} may be obtained. 
However such a result is of limited use since, although with respect to a stronger norm, the result only holds for operators defined on $\operatorname{Dom}(Z)$. 
 It is clear that we would pay a price if we restrict our attention to  $  \dmeasures_{\ddagger} $ as apposed to $ \dmeasures$.   The problem is that the space $ \dmeasures_{\ddagger}$ is dependent on the dynamics. I.e. if we considered another   semiflow acting on the same branched manifold the space $ \dmeasures_{\ddagger}$ would be different. This is of course a problem if one is interested in studying perturbations of the semiflow. If we are using $\dmeasures$ we may study all flows defined  on $\Oo$ with the same Banach space.
One of the main themes of this exposition is to demonstrate that all we wish to know about the flow can be recovered without resorting to studying the generator $Z$.
It is tempting to think that this difficulty is merely the fault of a poor choice of Banach space to start with. However this is not the case for the semiflows we are considering. The range of the resolvent will always consist of measures which have densities that are smooth along flow lines whilst the branches of the manifold means that we would expect jumps in the densities if one considers sections transversal to the flow lines.

\begin{thmm}\label{thm:details}
\mbox{ }
\begin{enumerate}
\item 
 Suppose $\Re(z)>0$. Then ${z^{-1}}$ is in the spectrum of $\R{z}:\dmeasures \to \dmeasures$. If $\abs{z} <\Re(z) +\lambda$ then
 ${z^{-1}}$ is an eigenvalue for $\R{z}:\dmeasures \to \dmeasures$.
 \item
Suppose $\Re(z)>0$, $\Re(w)>0$ and $\mu$ is an eigenvector for $\R{z}$ corresponding to the eigenvalue $z^{-1}$.  Then $\mu$ is  an eigenvector for $\R{w}$ corresponding to the eigenvalue $w^{-1}$.
 \item
For all $t>0$ then $1$ is an eigenvalue for $\flo{t}_{*}:  \dmeasures \to \dmeasures$ (perhaps not isolated). Moreover each eigenvector corresponding to this eigenvalue is also an eigenvector  for $\R{z}$ corresponding to the eigenvalue $z^{-1}$ where $\Re(z)>0$.
\item
  The function  $z\mapsto \R{z}$ taking values in $\cB(\dmeasures, \dmeasures)  $ has a pole at $z=0$.
    \end{enumerate}
\end{thmm}
\begin{proof}   
Fix $z\in\bC$ such that $\Re(z)>0$. We define the linear functional $\ell \in \dmeasures^{*}$ by setting $\ell(\mu):= \mu(1)$ for all $\mu\in\dmeasures$. We calculate, for all $\mu\in \dmeasures$, that
\[
\begin{split}
\left(  {\R{z}}^{*} \ell \right)\mu 
&= \ell(\R{z}\mu)   =\left(  \R{z} \mu \right)(1)\\
&= \int_{0}^{\infty} e^{-zt} \left( \flo{t}_{*}\mu\right)(1) \ dt\\
 &= \int_{0}^{\infty} e^{-zt}\mu(1) \ dt
 = z^{-1}\mu(1) 
 = z^{-1} \ell(\mu).
\end{split}
\]
This means that $\ell$ is an eigenvector for ${\R{z}}^{*}$ corresponding to the eigenvalue $z^{-1}$. Consequently $z^{-1} \in \operatorname{Spec}(\R{z})$ and by the quasi-compactness result of Proposition~\ref{prop:quasicompact} we know that $z^{-1}$ is actually  an eigenvalue and so we prove item~(1) of the theorem. 
Suppose now that $\mu\in \dmeasures$ is an eigenvector for $\R{z}$ corresponding to the eigenvalue $z^{-1}$. Since $\R{z}$ satisfies the resolvent equation (Lemma~\ref{lem:resolventequation})
\[
\begin{split}
z \R{w}\mu - \mu 
&= z[ \R{w} - \R{z} ] \mu\\
&= z(z-w)\R{w}\R{z} \mu 
= (z-w) \R{w} \mu,
\end{split}
\]
and so $\R{w}\mu = w^{-1}\mu$. This proves item~(2).
For all $\nu\in \dmeasures$ and $t>0$ we have (as in \cite[Problem 8.2.4]{Davies:2007qy} this may be shown as a consequence of the formula of Lemma~\ref{lem:Rzn})
\[
\flo{t}_{*}\nu = \lim_{n\to \infty}  \left( \frac{n}{t} \R{n/t}   \right)^{n}\nu.
\]
Applying the  result of item~(2) to the above formula we obtain immediately that for all $n\in \bN$ we have $ \frac{n}{t} \R{n/t} \mu = \mu$ and consequently that $\flo{t}_{*}\mu  = \mu$ as required to prove item~(3). Item~(4) is now obvious from the definition of $\R{z}$.
\end{proof}

We now give some indication of how to relate the above one-parameter semigroup theory of Theorem~\ref{thm:inverse}, Theorem~\ref{thm:extension} and Theorem~\ref{thm:details} to the statistical properties of the flow as has been developed in \cite{BL}.
Theorem~\ref{thm:details} allows us to conclude immediately that there is at least one invariant measures in $\dmeasures$ and no more than a finite number. This in turn leads to the ergodic decomposition of the dynamical system.
The pole of $\R{z}$ at $z=0$ is simple if and only if the absolutely continuous invariant measure is unique. The flow $\flo{t} : \Oo \to \Oo$ is mixing if and only if $0$ is the only singularity of $\R{z}$ on the imaginary axis. 
It is convenient to let $\cG$ denote the set of $z\in \bC$ such that $\Re(z)>-\lambda$ and that $z$ is a pole of $z\mapsto \R{z} \in \cB(\dmeasures,\dmeasures)$. 
 $\cG \cap i \bR$ is a group and the associated eigenfunctions are all measures absolutely continuous with respect to a convex combination of the absolutely continuous invariant measures.
By Theorem~\ref{thm:extension} we know that $z\mapsto \R{z}$ is holomorphic on $\{z\in \bC: \Re(z)>-\lambda\} \setminus \cG$.
For each $z\in \cG$ let
\[
\Pi_{z}:= \frac{1}{2\pi i} \int_{\gamma} \R{z} \ dz
\]
where $\gamma$ is a positively-orientated small circle enclosing $z$ but excluding all other singularities of $\R{z}$. As with spectral projectors the resolvent equation, proven in Lemma~\ref{lem:resolventequation}, implies that the definition is  independent on the choice of $\gamma$ subject to the above conditions.
In this way the operator-theoretic  results  of Theorem~\ref{thm:inverse}, Theorem~\ref{thm:extension} and Theorem~\ref{thm:details} can be used to understand the fine statistical properties of the flow, even in settings where the discontinuities of the system present considerable obstacles to the study of perturbations.

%
%
%
%

This present section was completely self-contained apart from the proofs of Proposition~\ref{prop:bounded} and Proposition~\ref{prop:quasicompact} which are the contents of  Section~\ref{sec:bounded} and Section~\ref{sec:proofQC}  respectively.
Section~\ref{sec:branchedmanifolds} contains the definition and notation related to branched manifolds.

\section{Branched Manifolds}\label{sec:branchedmanifolds}
The definition of a branched manifold we use here is that given  by Williams~\cite{williamsexpanding} where they are shown to arise from quotients of dynamical foliations for expanding attractors. We now recall  the definition. 
\begin{defin}[Branched Manifold]\label{def:branched}
A  {$d$-dimensional branched manifold} of class $\cC^r$ is a metrizable space $\Oo$ together with:
\begin{enumerate}
 \item A  countable collection $\{\U_i\}_{i \in \Ichart }$ of closed subsets of $\Oo$ and for each $i$ a map $\var_i:\U_i \to \D_{i}$ where $\D_{i}$ is a closed $d$-ball in $\bR^{d}$.
 \item  A  countable  collection $\{\V{i}{j}\}_{j\in\Isubchart_{i}}$ of closed subsets of $\U_i$, for each $i$.
   \setcounter{saveenum}{\value{enumi}}
\end{enumerate}
Subject to the following axioms:
\begin{enumerate}
  \setcounter{enumi}{\value{saveenum}}
 \item $\displaystyle\bigcup_{j\in \Isubchart_{i}} \V{i}{j} = \U_i$ for each $i$.
 \item $\displaystyle\bigcup_{i\in\Ichart} \Int{\U_{i}} = \Oo$.
\item   For each $i\in \Ichart,j\in\Isubchart_{i}$ the map ${\var_{i}|}_{\V{i}{j}}$  (i.e.  $\var_i$ restricted to $\V{i}{j}$)  is a homeomorphism onto its image $\var_{i}(\V{i}{j})\subseteq \D_{i}$ and this image is a closed subset of $\D_{i}$.
\item  For each $i,i'\in \Ichart$ there exists a $\cC^r$-diffeomorphism $\al_{ii'}$ with domain $\var_{i'}(\U_i \cap \U_{i'})$ such that $\var_i = \al_{ii'} \circ \var_{i'}$ when defined.
\end{enumerate}
\end{defin}
\noindent The sets $\U_i$ are called the charts, the sets $\V{i}{j}$ are called the subcharts. If it is possible to cover each $\U_i$ with just one subchart, i.e. $\{\V{i}{j}\}_j = \{\U_i\}$, the above definition reduces to that of a manifold with boundary but without branches.  The $\var_{i}$ are called coordinate maps and the $\al_{ii'}$ are called transition maps.
The branched manifold $\Oo$ contains both interior points and boundary points which are defined as follows. 
\begin{defin}[Interior Points]
A point $p\in \Oo$ is said to be an \emph{interior point} of the branched manifold if there exists $i\in \Ichart$, a set $ G\subset \U_{i}$ and $\delta>0$ such that $p\in G$ and  $\var_{i}(G) =\{ \var_{i}(p)+ y: y \in \bR^{d}, \abs{y}< \delta\}$.
\end{defin}
\begin{defin}[Boundary Points]
The complement of the interior points are called the  \emph{boundary points}. We let $\partial\Oo$ denote the set of all boundary points.
\end{defin}
%
%
%
\begin{defin}[Differentiable]
 The space $\cC^k(\Oo)$, for each $k\in\{1,2,\ldots, r\}$, is defined as the set of all maps $f:\Oo \to \bC$ such
that for every $i\in \Ichart$ and $j\in \Isubchart_i$ the map
\[
f\circ \left(\smash{{\var_{i}|}_{\V{i}{j}}}\right)^{-1} : \var_{i} (\V{i}{j})  \to \bC 
\quad \quad \text{is of class $\cC^{k}$}.
\]
\end{defin}

\begin{defin}[Tangent Bundle]
For each $i\in \Ichart$, $j\in \Isubchart_{i}$ we have the induced bundle over $\V{i}{j}$ given by 
$ \left(\smash{{\var_{i}|}_{{\V{i}{j}}}}\right)^{*} T\bR^{d}$.
Consider the disjoint union
\[
\bigsqcup_{i\in \Ichart, j\in \Isubchart_i} \left(\smash{{\var_{i}|}_{{\V{i}{j}}}}\right)^{*} T\bR^{d}
=  \left\{ (x,v,i,j):x\in \V{i}{j}, v \in T_{\var_{i}(x)}\bR^{d}, i\in \Ichart, j\in \Isubchart_i \right\}.
\]
We   introduce the  relation which sets $(x,v,i,j)\sim (x',v',i',j')$ if 
  $x = x'$ and  also
  ${(D\alpha_{i' i}) v = v'}$.
  The tangent bundle over $\Oo$, written  $T\Oo$, is defined as the above disjoint union subject to this equivalence relation.
\end{defin}

\begin{defin}[Foliation]
Suppose that there exists coordinate charts $\var_{i}: \U_{i} \to \bR^{d}$  such that the transition maps
$
\al_{ii'} : \var_{i'}(\U_{i} \cap \U_{i'}) \to \var_{i}(\U_{i} \cap \U_{i'})
$
which satisfy $\var_{i} = \al_{ii'}\circ \var_{i'}$ are of the form
\begin{equation}\label{eq:john}
\al_{ii'}(x,y) = \left( \al_{ii'}^{(1)}(x) , \al_{ii'}^{(2)}(x,y)   \right),
\end{equation}
where $x$ represents $n$ coordinates and $y$ represents $d-n$ coordinates.
For all $c\in \bR^{n}$ let $F_{c}:= \{p\in \U_{i}: \var_{i}(p) = (c, y), y\in \bR^{d-n}\}   $. These stripes are called the \emph{plaques} of the foliation. By \eqref{eq:john} these plaques match up from chart to chart to form the \emph{leaves} of the $n$-dimensional  foliation $\cF$.
\end{defin}


\section{$\flo{t}_{*}:\dmeasures \to \dmeasures $ is Bounded}   
\label{sec:bounded}
In this section we show that the operators $\flo{t}_{*}:\dmeasures \to \dmeasures $ are bounded and so prove Proposition~\ref{prop:bounded}. 
Recall that $X$ is the vector field associated to the flow and $V$ is a unit vector field transversal to $X$.
For each $\mu \in \dmeasures$ let
\begin{equation}\label{eq:defnormdag}
\norm{\mu}_{\widetilde\dmeasures}:= \TotalVar{ D_{V} \mu } + \TotalVar{ D_{X} \mu } + \TotalVar{ \mu}.
\end{equation}
This  defines a norm on $\dmeasures$, and  importantly it has the following property.
\begin{lem}\label{lem:equivalent}
The norms $\norm{\cdot}_{\dmeasures} $ and $\norm{\cdot}_{\widetilde\dmeasures}$ are equivalent on $\dmeasures$.
\end{lem}
\begin{proof}
Since $V$ and $X$ are uniformly transversal there exists $C<\infty$ such that any vector field $U$, $\supnorm{U}\leq 1$ may be written as $U= \alpha V + \beta X$ where $\supnorm{\alpha}\leq C$ and $\supnorm{\beta}\leq C$. This means that $\TotalVar{D_{U}\mu} \leq C \norm{\mu}_{\widetilde\dmeasures}$ for all $\mu\in \dmeasures$.
The other direction in immediate.
\end{proof}
Recall the quantity $\lambda>0$ given by the uniform expansion assumption. 
\begin{lem}\label{lem:split}
Suppose $t\geq 0$ and let $\Oo_{t} := \Oo \setminus (\flo{t})^{-1}\partial \Oo$.   There exists $A_{t}\in \cC^{1}(\Oo_{t})$ and $ B_{t}\in \cC^{1}(\Oo_{t})$  
    such that on $\Oo_{t}$ we have 
\begin{equation}\label{eq:splitting}
\begin{split}
X\eta \circ \flo{t} &= X\left( \eta \circ \flo{t} \right),\\
V\eta \circ \flo{t} &= A_{t} V\left( \eta \circ \flo{t} \right) + B_{t} X\left( \eta \circ \flo{t} \right) 
\end{split}
\end{equation}
for all $\eta \in \cC^{1}_{0}(\Oo)$. Moreover  exists  $C <\infty$, $\alpha<\infty$ such that $\supnorm{A_{t}} \leq C e^{-\lambda t}$, $\supnorm{B_{t}}\leq C$,
$\supnorm{VA_{t}}\leq Ce^{\alpha t}$ and  
$\supnorm{XB_{t}}\leq Ce^{\alpha t}$
  for all $t\geq 0$.
\end{lem}
\begin{proof}
The first line of \eqref{eq:splitting} is nothing more than the observation that $X$ is the vector field associated to the flow $\flo{t}$ and so is invariant under the action of the flow.
Since the vector fields $V$ and $X$ are transversal it is always possible to write $V\eta \circ \flo{t}$ of the form given in the second line of \eqref{eq:splitting}.  And since $\flo{t}:\Oo\to \Oo$ is $\cC^{2}$ we know that $A_{t}$ and $B_{t}$ are $\cC^{1}$ on the set $ \Oo \setminus (\flo{t})^{-1}\partial \Oo$. Fixing $p\in \Oo$ and using the vector fields $V$ and $X$ (respectively) as a basis for tangent space at that point we can write 
$D_{p}\flo{t} = \left(\begin{smallmatrix}\alpha_t & 0 \\\beta_t & 1\end{smallmatrix}\right)$
and by the uniform expansion assumption we know that $\abs{\alpha_{t}} \geq C^{-1}e^{\lambda t}$. Taking the inverse of the matrix we have that $A_{t}(p) = 1/\alpha_{t}$ and so $\supnorm{A_{t}} \leq C e^{-\lambda t}$ as required.
We continue to use the vector fields $V$ and $X$ respectively as a basis for tangent space. Let $t>0$, $n\in \bN$, $\tau:= t / n$, and  $p_{j}:= \flo{j \tau}p$ for all $t\in \{0,1,\ldots,n\}$.  We may write
\begin{equation}\label{eq:jack}
\begin{split}
\left( D_{p}\flo{t}  \right)^{-1}
&=  \left(\begin{smallmatrix}A_t & 0 \\ B_t & 1\end{smallmatrix}\right)(p)\\
&=  \left(\begin{smallmatrix}A_\tau & 0 \\ B_\tau & 1\end{smallmatrix}\right)(p_{n-1})
 \left(\begin{smallmatrix}A_\tau & 0 \\ B_\tau & 1\end{smallmatrix}\right)(p_{n-2})
 \cdots
  \left(\begin{smallmatrix}A_\tau & 0 \\ B_\tau & 1\end{smallmatrix}\right)(p).
  \end{split}
  \end{equation}
  The idea is that for any $t>0$ we will always choose $n\in \bN$ such that $\tau \in (0,1)$.
  The above product of matrices formula means that
  \begin{equation}\label{eq:Bt}
  \begin{split}
   B_t(p) &= B_{\tau}(p)\\
   & \ \ + B_{\tau}(p_{1})A_{\tau}(p)\\
   & \ \    +\ldots + \ldots \\
   & \ \
   + B_{\tau}(p_{n-1})A_{\tau}(P_{n-1})\cdots A_{\tau}(p_{1})A_{\tau}(p).
   \end{split}
\end{equation}
Combined with the already proven estimate  $\supnorm{A_{t}} \leq C e^{-\lambda t}$  the above geometric sum gives the uniform (in $t$) bound for $\supnorm{B_{t}}$. We increase the value of $C$ as required so that $\supnorm{B_{t}}\leq C$ for all $t\geq 0$.
From \eqref{eq:jack} we know that the quantities $\supnorm{VA_{t}}$ and  
$\supnorm{XB_{t}}$ cannot grow faster than some exponential rate and so we may choose some $\alpha<\infty$ as required by the statement of the lemma.
\end{proof}

\begin{rem}\label{rem:markov}
Historically the Markov property was important in the study of dynamical systems.
We never required any such property for the flow studied in this work. 
In the branched manifold setting we say that $\flo{t}:\Oo \to \Oo$ is Markov if $\flo{t}\partial \Oo \subset \partial \Oo$ or more generally that there exists some zero measure set $S \supset \partial \Oo$ such that $\flo{t}S \subset S$.  However we may always increase the boundary of the branched manifold by adding any piece of flowline without changing any of the properties of the flow. Therefore without loss of generality we may always consider the first of the above statements.   
Without the Markov property the technical problem  is that $\eta\in \cC_{0}(\Oo)$ does not imply that $\eta \circ \flo{t}\in \cC_{0}(\Oo)$ since there is now no reason to expect $\eta\circ\flo{t}(p) = 0$ for all $p\in \partial \Oo$. This is a reason why we must use bounded variation type norms in the present setting and not $\cC^{1}$ type norms.
\end{rem}

Shortly we will require the following lemma.
\begin{lem}\label{lem:zeroed}
Suppose that  $\eta\in \cC^{1}(\Oo)$, $\supnorm{\eta}\leq 1$, $\mu\in \dmeasures$ and $\epsilon>0$. Then there exists $\tilde \eta \in \cC^{1}_{0}(\Oo)$ such that $\supnorm{\tilde\eta}\leq 1$ and $\abs{\mu(X[\eta-\tilde\eta])}\leq \epsilon$.
\end{lem}
\begin{proof} 
If $\partial \Oo = \emptyset$ we may take $\tilde\eta = \eta$.
Otherwise this lemma is a consequence of $X$ being tangent to $\partial \Oo$.
Let $\cF_{V}\!\operatorname{-dist}$ denote distance on $\Oo$ restricted to the leaves of $\cF_{V}$.
Let $\Delta>0$ and $\delta_{0}>0$ be sufficiently small, to be chosen later.
For all $\delta \in (0,\delta_{0})$ let
\[
S_{\delta} := \{p\in \Oo : \cF_{V}\!\operatorname{-dist}(p,\partial\Oo) \leq \delta\}.
\]
And let
\[
S_{\delta,\Delta}:= S_{\delta} \cup \{p\in \Oo:\flo{t}p\in S_{\delta}, t\in[0,\Delta]\}
\cup \{\flo{t}p: p\in S_{\delta}, t\in[0,\Delta]\}.
\]
Furthermore let $\omega_{\delta}:\Oo \to [0,1]$ be such that $\omega_{\delta}\in \cC_{0}^{1}(\Oo)$ (in particular $\omega_{\delta}(p)=0$ for all $p\in \partial\Oo$),  $\omega_{\delta}(p)=1$ for all $p\in\Oo\setminus S_{\delta,\Delta}$ and that there exists $C<\infty$ such that $\supnorm{X\omega_{\delta}}\leq C$ for all $\delta\in [0,\delta_{0}]$. 

Let $\eta_{\delta}:= \omega_{\delta}\cdot \eta \in \cC_{0}^{1}(\Oo)$. 
We must estimate $\abs{ \mu(X(\eta - \eta_{\delta}))  } = \abs{ \mu(X([1-\omega_{\delta}]\eta))  } $.
For all $ \delta \in (0,\delta_{0}) $ we have $\abs{X([1-\omega_{\delta}]\eta)} \leq C+\abs{X\eta}$ and also the support of $X([1-\omega_{\delta}]\eta)$ is contained within $S_{\epsilon,\delta}$. Since, as noted before, $\mu\in \dmeasures$ has densities of bounded variation and in particular the densities are $\mathbf{L^{1}}$ this means that $ \abs{ \mu(X([1-\omega_{\delta}]\eta))  } \to 0$ as $\delta \to 0$.
\end{proof}

\begin{lem}\label{lem:Xest}
For all $\mu\in \dmeasures$ and $t\geq 0$
\[
\TotalVar{\D_{X}(\flo{t}_{*}\mu)} \leq \TotalVar{\D_{X}\mu}.
\]
\end{lem}
\begin{proof}
Fix $\mu\in \dmeasures$ and $\eta\in \cC_{0}^{1}(\Oo)$. By Lemma~\ref{lem:split} we have that for all $t\geq0$
\[
\D_{X}(\flo{t}_{*}\mu)(\eta) = \mu(X\eta\circ \flo{t}) = \mu(X(\eta\circ\flo{t})).
\]
The problem  is that $\eta\in \cC_{0}^{1}(\Oo)$ does not imply that $\eta \circ \flo{t}\in \cC_{0}^{1}(\Oo)$ since we do not require the Markov property as discussed in Remark~\ref{rem:markov}. However by Lemma~\ref{lem:zeroed}, for all $\epsilon>0$ there exists $\zeta \in \cC^{1}_{0}(\Oo)$ such that $\abs{ \mu(X(\zeta - \eta\circ \flo{t}))  }\leq \epsilon$.
\end{proof}

\begin{lem}\label{lem:linearpart}
 There exists $C<\infty$ such that for all  $t\geq0$ and $\eta\in \cC^{1}_{0}(\Oo)$ with $\supnorm{\eta}\leq 1$ there exists $\Psi_{\eta,t}\in \cC^{1}(\Oo)$    such that $(\eta\circ \flo{t} - \Psi_{\eta,t}) \in \cC_{0}(\Oo)$, $\supnorm{\Psi_{\eta,t}}\leq 1$ and  $\supnorm{V\Psi_{\eta,t}}\leq C$.
\end{lem}
\begin{proof}
Recall that by assumption there exists the foliation $\cF_{V}$ whose leaves are all curves of length at least $\delta>0$ and with end points contained within $\partial\Oo$. We therefore define $\Psi_{\eta,t}$ to be equal to $\eta \circ \flo{t} $ on $\partial\Oo$ and linear along the leaves of $\cF_{V}$. The uniform minimum length of these curves gives the uniform bound for $\supnorm{V\Psi_{\eta,t}}$.
\end{proof}

Recall the quantity $\lambda>0$ given by the uniform expansion assumption and the quantity $\alpha<\infty$ given by Lemma~\ref{lem:split}.
\begin{lem}\label{lem:Vest}
 There exists $C<\infty$ such that, for all $\mu\in \dmeasures$ and $t\geq 0$
\[
\TotalVar{\D_{V}(\flo{t}_{*}\mu)} \leq Ce^{-\lambda t}\norm{D_{V}\mu}_{\measures}
+ C\norm{D_{X}\mu}_{\measures} + Ce^{\alpha t}\norm{\mu}_{\measures}.
\]
\end{lem}
\begin{proof}
Fix $\mu\in \dmeasures$,  $t\geq0$ and $\eta\in \cC_{0}^{1}(\Oo)$. By Lemma~\ref{lem:split} we have that 
\begin{equation}\label{eq:sally}
\left(\D_{V}(\flo{t}_{*}\mu)\right)(\eta) 
=\mu(V\eta \circ \flo{t})
=\mu(A_{t} \cdot V\left( \eta \circ \flo{t} \right)) + \mu(B_{t} \cdot X\left( \eta \circ \flo{t} \right) ).
\end{equation}
We will estimate these two term separately. First we estimate $\abs{ \smash{\mu(A_{t} \cdot V\left( \eta \circ \flo{t} \right))  }}$. Using the quantity $\Psi_{\eta,t}$ defined in  Lemma~\ref{lem:linearpart} we have
\[
\mu(A_{t} \cdot  V\left( \eta \circ \flo{t} \right)) 
= \mu( V\left( A_{t}[\eta \circ \flo{t} - \Psi_{\eta,t} ]\right)) 
+ \mu(V(A_{t} \cdot   \Psi_{\eta,t}) ) 
-\mu(VA_{t} \cdot \eta \circ \flo{t} ) 
\]
Note that $A_{t}[\eta \circ \flo{t} - \Psi_{\eta,t} ] \in \cC_{0}(\Oo)$ and $\abs{\smash{A_{t}[\eta \circ \flo{t} - \Psi_{\eta,t} ]}} \leq Ce^{-\lambda t}$ by Lemma~\ref{lem:split}. This means that $\abs{\mu( V\left( A_{t}[\eta \circ \flo{t} - \Psi_{\eta,t} ]\right)) } \leq 2  Ce^{-\lambda t} \norm{D_{V}\mu}_{\measures}$. The second and third terms are bounded by $C e^{\alpha t}\norm{\mu}_{\measures}$ by the estimates of Lemma~\ref{lem:split}.
This means that
\begin{equation}\label{eq:mice}
\abs{  \mu(A_{t} \cdot  V\left( \eta \circ \flo{t} \right))  }
\leq Ce^{-\lambda t} \norm{D_{V}\mu}_{\measures}
+ C e^{\alpha t}\norm{\mu}_{\measures}.
\end{equation}
Now we estimate $\abs{ \smash{\mu(B_{t} X\left( \eta \circ \flo{t} \right) ) } }  $, the second term of \eqref{eq:sally}. We observe that
$\mu(B_{t} X\left( \eta \circ \flo{t} \right) )  = \mu( X\left( B_{t} \cdot\eta \circ \flo{t} \right) ) -   \mu(X B_{t} \cdot  \eta \circ \flo{t}  ) $.
By the same reasoning as the proof of Lemma~\ref{lem:Xest}, using also Lemma~\ref{lem:zeroed}, we have that
\begin{equation}\label{eq:cat}
\abs{  \mu(B_{t} \cdot  X\left( \eta \circ \flo{t} \right))  }
\leq C \norm{D_{X}\mu}_{\measures}
+ C e^{\alpha t}\norm{\mu}_{\measures}.
\end{equation}
By \eqref{eq:sally}, the estimates of \eqref{eq:mice} and \eqref{eq:cat} complete the proof of the lemma.
\end{proof}

\begin{proof}[Proof of Proposition~\ref{prop:bounded}]
The equivalence of the norms by Lemma~\ref{lem:equivalent} and the estimates of Lemma~\ref{lem:Xest} and Lemma~\ref{lem:Vest} complete the proof of Proposition~\ref{prop:bounded}. 
\end{proof}

\section{Essential Spectral Radius of $\R{z}:\dmeasures \to \dmeasures$}
\label{sec:proofQC}

In this section we prove Proposition~\ref{prop:quasicompact}. Recall the quantity $\alpha<\infty$ which was given by Proposition~\ref{prop:bounded} and which had its origin in the estimates of Lemma~\ref{lem:split}.

\begin{lem}\label{lem:boundRz3}
For all $n\in \{2,3,\ldots\}$,  $z\in \bC$ such that $\Re(z)>\alpha$ and $\mu\in \dmeasures$
\[
\TotalVar{D_{X}\R{z}^{n}\mu} \leq  (\abs{z} + \Re(z)) \  \Re(z)^{-n} \TotalVar{\mu}. 
\] 
\end{lem}
\begin{proof}
Fix $\eta \in \cC_{0}^{1}(\Oo)$ such that $\supnorm{\eta} \leq 1$ and $z\in \bC$ such that $\Re(z)>\alpha$.
Using the formula from Lemma~\ref{lem:Rzn} we have 
\begin{equation}\label{eq:alice}
\begin{split}
D_{X}\R{z}^{n}\mu(\eta) &=   \frac{ -1 }{(n-1)!} \int_{0}^{\infty}  { t^{n-1}e^{-zt}} \mu(X\eta \circ \flo{t} ) \ dt \\
&=  \frac{ -1 }{(n-1)!}  \mu\left( \int_{0}^{\infty}  { t^{n-1}e^{-zt}}  \ X\eta \circ \flo{t} \ dt \right). 
\end{split}
\end{equation}
Since $X\eta \circ \flo{t} = X\left( \eta \circ \flo{t} \right) = \frac{d}{dt}(\eta \circ \flo{t}) $ by Lemma~\ref{lem:split} 
and integrating by parts
\[
 \int_{0}^{\infty}   t^{n-1}e^{-zt}  \ X\eta \circ \flo{t} \ dt 
 = - \int_{0}^{\infty}   \frac{d}{dt}(  t^{n-1} e^{-zt})   \ \eta \circ \flo{t} \ dt.
\]
There are no boundary terms in the integration by parts since for each $p\in \Oo$ the map $t\mapsto \eta\circ \flo{t}(p)$ is continuous and $t^{n-1}e^{-zt} \to 0$ as $t\to0$ and as $t\to \infty$. Substituting the above into \eqref{eq:alice} and noting that $\abs{\mu(  \eta \circ \flo{t} )} \leq \TotalVar{\mu}$ as discussed in the proof of Lemma~\ref{lem:TotalVarBound} we have
\[
\begin{split}
\abs{D_{X}\R{z}^{n}\mu(\eta)}
& \leq   \frac{ 1}{(n-1)!}  \int_{0}^{\infty}  \abs{ \frac{d}{dt}(  t^{n-1} e^{-zt})}   \ \abs{\mu( \eta \circ \flo{t})} \ dt\\
& \leq   \frac{ 1}{(n-1)!}  \left( \int_{0}^{\infty}  \abs{ \frac{d}{dt}(  t^{n-1} e^{-zt})}   \ dt \right) \  \TotalVar{\mu}.
\end{split}
\]
It remains to calculate the integral. 
Since $  \tfrac{d}{dt} \left(t^{n-1}e^{-zt}  \right)  = (n-1)t^{n-2}e^{-zt}- zt^{n-1}e^{-zt} $
then $  \abs{ \frac{d}{dt}(  t^{n-1} e^{-zt})}   \leq  (n-1)t^{n-2}e^{-\Re(z)t} +  \abs{z}t^{n-1}e^{-\Re(z)t}$.
 For each $m\in \{1,2,\ldots\}$ and $a>0$ then $\int_{0}^{\infty} t^{m-1} e^{-a t} \ dt = (m-1)! \ a^{-m}$ and so 
\begin{equation}\label{eq:nina}
 \frac{ 1}{(n-1)!}  \int_{0}^{\infty}  \abs{ \frac{d}{dt}(  t^{n-1} e^{-zt})}   \ dt
 \leq  (\abs{z} + \Re(z)) \  \Re(z)^{-n}. 
 \end{equation}
\end{proof}

\begin{lem}\label{lem:boundRz2}
There exists $C<\infty$ such that for all $n\in \{2,3,\ldots\}$
\[\begin{split}
\TotalVar{D_{V}\R{z}^{n}\mu} &\leq C (\Re(z)+\lambda)^{-n}\TotalVar{D_{V}\mu}\\ 
& \ \  + C (\Re(z)-\alpha)^{-n}\TotalVar{\mu} + C(\abs{z}+\Re(z))\Re(z)^{-n}   \TotalVar{\mu}.
\end{split}\]
 for all $z\in \bC$ such that $\Re(z)>\alpha$ and $\mu\in \dmeasures$. 
\end{lem}
\begin{proof}
Fix $\eta \in \cC_{0}^{1}(\Oo)$ such that $\supnorm{\eta} \leq 1$ and $z\in \bC$ such that $\Re(z)>\alpha$.
Using the formula from Lemma~\ref{lem:Rzn}  we have 
\begin{equation}\label{eq:jane}
\begin{split}
D_{V}\R{z}^{n}\mu(\eta) &=  \frac{ -1}{(n-1)!}  \int_{0}^{\infty}  t^{n-1}e^{-zt} \mu(V\eta \circ \flo{t} ) \ dt \\
&=  \frac{ -1}{(n-1)!}  \mu\left( \int_{0}^{\infty} t^{n-1}e^{-zt}  \ V\eta \circ \flo{t} \ dt \right). 
\end{split}\end{equation}
Using  Lemma~\ref{lem:split} we have $ V\eta \circ \flo{t} = A_{t} \cdot V\left( \eta \circ \flo{t} \right) + B_{t}\cdot X(\eta \circ \flo{t})$ and so
\[\begin{split}
 \int_{0}^{\infty} t^{n-1}e^{-zt}  V\eta \circ \flo{t} \ dt
 &=   \int_{0}^{\infty} t^{n-1}e^{-zt}  A_{t} \cdot  V\left( \eta \circ \flo{t} \right)        \ dt\\
& \ \  + \int_{0}^{\infty}  t^{n-1}e^{-zt}  B_{t}   \cdot X(\eta \circ \flo{t})      \ dt.
\end{split}\]
For the first term of the right hand side, since $A_{t} \in \cC^{1}(\Oo)$, we use that $A_{t}  \cdot V\left( \eta \circ \flo{t} \right)   = V\left(  A_{t}   \cdot \eta \circ \flo{t} \right)   -  V A_{t}  \cdot \eta \circ \flo{t}   $. 
Notice that $X\left( \eta \circ \flo{t} \right) = \frac{d}{dt}(\eta \circ \flo{t}) $ and 
that $B_{t}   \cdot X(\eta \circ \flo{t})   =  X(B_{t}   \cdot \eta \circ \flo{t})   - X B_{t}   \cdot (\eta \circ \flo{t})$.
We use this for the second term of the right hand side and we integrate by parts, as in the proof of Lemma~\ref{lem:boundRz3}. Since for each $p\in \Oo$ the functions $t\mapsto B_{t}(p)$ and $t\mapsto \eta\circ\flo{t}(p)$ are continuous  we have pointwise on $\Oo$
\begin{equation*}
\begin{split}
\int_{0}^{\infty} { t^{n-1}e^{-zt}} B_{t} \cdot X(\eta \circ \flo{t})      \ dt
&=   \int_{0}^{\infty}  \tfrac{d}{dt} \left(t^{n-1}e^{-zt}\right) \cdot  B_{t}  \cdot   \eta \circ \flo{t}  \ dt \\
& \ \ - \int_{0}^{\infty} t^{n-1}e^{-zt} XB_{t} \cdot \eta\circ \flo{t} \ dt.
\end{split}
\end{equation*}
There are no boundary terms in the integration by parts since $\Re(z)>\alpha$ and so $    \abs{ t^{n-1}e^{-zt}  }  \to 0$ both as $t\to 0$ and as $t\to \infty$. 
So, collecting together the above calculations, we have shown that
\begin{multline*}
\begin{split}
 \int_{0}^{\infty} { t^{n-1}e^{-zt}} V\eta \circ \flo{t} \ dt
 &=    \int_{0}^{\infty} { t^{n-1}e^{-zt}}   V\left(  A_{t} \cdot (  \eta \circ \flo{t} - \Psi_{\eta,t} )\right)    \ dt\\
 & \ \ +  \int_{0}^{\infty} { t^{n-1}e^{-zt}} \left[  V(A_{t}\cdot \Psi_{\eta,t})   - (V A_{t} + XB_{t})  \cdot \eta \circ \flo{t}\right] \ dt \\    
   & \ \ +  \int_{0}^{\infty} {   \tfrac{d}{dt} \left(t^{n-1}e^{-zt}  \right)} \cdot B_{t}   \cdot    \eta \circ \flo{t}  \ dt,
\end{split}
\end{multline*}
where we have used that quantity $\Psi_{\eta,t}$ which was defined in Lemma~\ref{lem:linearpart}.
Recalling \eqref{eq:jane} this means that 
\begin{equation}\label{eq:sparrow}
\begin{split}
\abs{D_{V}\R{z}\mu(\eta) }
&\leq \frac{1}{(n-1)!}    \int_{0}^{\infty} t^{n-1}e^{-\Re(z) t}\abs{ D_{V}\mu( A_{t} \cdot  (  \eta \circ \flo{t} - \Psi_{\eta,t} ) )  } \ dt \\
& \ \ +  \frac{1}{(n-1)!}  \int_{0}^{\infty} { t^{n-1}e^{-zt}} \abs{\mu(V(A_{t}\cdot \Psi_{\eta,t})   - (V A_{t} + XB_{t})  \cdot \eta \circ \flo{t}    )}   \ dt \\
& \ \ +  \frac{1}{(n-1)!}    \int_{0}^{\infty} \abs{   \tfrac{d}{dt} \left(t^{n-1}e^{-zt}  \right)}    \abs{  \mu (B_{t}  \cdot    \eta \circ \flo{t})   }  \ dt.
\end{split}
\end{equation}
That $(A_{t} \cdot  (  \eta \circ \flo{t} - \Psi_{\eta,t}) ) \in \cC^{1}_{0}(\Oo) $  by Lemma~\ref{lem:linearpart} and the other estimates    from Lemma~\ref{lem:split} we know that 
\[
\abs{ D_{V}\mu( A_{t} \cdot  (  \eta \circ \flo{t} - \Psi_{\eta,t} ) )  }  \leq C e^{-\lambda t} \TotalVar{D_{V}\mu}.
\] 
Furthermore 
\[
\abs{\mu(V(A_{t}\cdot \Psi_{\eta,t})   - (V A_{t} + XB_{t})  \cdot \eta \circ \flo{t}    )}  \leq C e^{\alpha t}  \TotalVar{\mu}.
\]
For the final term we have that $\abs{B_{t}  \cdot    \eta \circ \flo{t})  }\leq C$, also by  Lemma~\ref{lem:split}.
 Since $\TotalVar{  D_{V}\R{z}\mu } = \sup\left\{ \abs{D_{V}\R{z}\mu(\eta)}   :   \eta\in \cC^{1}(\Oo), \abs{\eta}_{\cC^{0}}\leq 1 \right\}$ substituting the above estimates in \eqref{eq:sparrow} and integrating, using also \eqref{eq:nina}, we obtain the estimate of the lemma.
\end{proof}

\begin{lem}\label{lem:estnorm2}
Exists $C<\infty$ such that for all  $\Re(z)>\alpha $, $n\in \{2,3,\ldots\}$ and $\mu\in \dmeasures$
\[
\norm{\R{z}^{n} \mu}_{\dmeasures} \leq C (\Re(z)+\lambda) ^{-n}  \TotalVar{D_{V}\mu} +C g_{z,n}   \TotalVar{\mu},
\]
where $g_{z,n}:= (\abs{z}+\Re(z))\Re(z)^{-n} + (\Re(z) - \alpha)^{-n}  $.
\end{lem}
\begin{proof}
We recall that $\norm{\R{z}^{n}\mu}_{\widetilde\dmeasures} = \TotalVar{D_{X}\R{z}^{n}\mu} +  \TotalVar{D_{V}\R{z}^{n}\mu} +  \TotalVar{\R{z}^{n}\mu} $ and combine the estimates from Lemma~\ref{lem:boundRz1}, Lemma~\ref{lem:boundRz3}, Lemma~\ref{lem:boundRz2}. We then recall that by Lemma~\ref{lem:equivalent}
the norms $\norm{\cdot}_{\dmeasures} $ and $\norm{\cdot}_{\widetilde\dmeasures}$  are equivalent.
\end{proof}

\begin{lem}\label{lem:compactness}
The embedding $\dmeasures \hookrightarrow \measures$ is compact.
\end{lem}
\begin{proof}
Any measure $\mu\in \dmeasures$ may be represented as  densities in the charts $\var_{i}(\V{i}{j}) \subset \bR^{2}$. These densities of of bounded variation.
This means that the lemma is a direct consequence of the classical result that
$\mathbf{BV}$ is compactly imbedded into $\mathbf{L^{1}}$.
\end{proof}

\begin{proof}
We follow Hennion's argument \cite{He}.
Fix $z\in \bC$ such that $\Re(z)> \alpha $ and for each $n\in \{2,3,\ldots\}$ let 
 \[
 B_{n}:= \{\R{z}^{n}\mu: \mu\in \dmeasures, \norm{\mu}_{\dmeasures}\leq 1\}
 \]
 and let $r_{n}$ denote the infimum of the $r$ such that the set $
B_{n}$
may be covered by a finite number of balls of radius $r$ (measured in the $\norm{\cdot}_{\dmeasures}$ norm). The formula of Nussbaum \cite{Nussbaum} states that 
  \begin{equation}\label{eq:nussbaum}
r_{ess}(\R{z})=\liminf_{n\to \infty}  \sqrt[n]{r_{n}}.
\end{equation}
By  Lemma~\ref{lem:compactness}
we know  that $B_{0}$ is relatively compact in the $\TotalVar{\cdot}$ norm and
therefore, for each $\epsilon >0$, there exists a finite set 
$\{G_{i}\}_{i=1}^{N_{\epsilon}}$ of  subsets of $B_{0}$ whose union covers $B_{0}$ 
and such that
\begin{equation}\label{compact1}
\TotalVar{\smash{\mu-\tilde \mu}} \leq \epsilon \quad \text{  for all $\mu,\tilde \mu \in G_{i}$}.
\end{equation}
 Notice that \( r_{n} \) can be bounded above by the supremum of the diameters of the elements of any given cover of \( B_{n} \). 
Since the union of \( \{G_{i}\}_{i=1}^{N_{\epsilon}} \) is a cover of \( B_{0} \),  then \( \{\R{z}^{n}(G_{i})\}_{i=1}^{N_{\epsilon}} \) is a cover of \( B_{n} \) and therefore it is  sufficient to obtain an upper bound for the maximum diameter of the \( P^{n}(G_{i}) \). We use 
 the estimate on $\norm{\R{z}^{n}\mu}_{\dmeasures}$ from  Lemma~\ref{lem:estnorm2}. This implies that for all $\mu,\tilde \mu \in G_{i}$ and $n\in \{2,3,\ldots\}$ then
 \begin{equation*}
\norm{ \smash{ \R{z}^{n}\mu- \R{z}^{n}\tilde \mu}}_{\dmeasures} \leq  C (\Re(z)+ \lambda)^{-n} \norm{\smash{\mu-\tilde \mu}}_{\dmeasures} + C_{z,n} \TotalVar{\smash{\mu-\tilde \mu}} . 
\end{equation*}
Substituting \eqref{compact1} we have shown that  
$r_{n}\leq  C(\Re(z) +\lambda)^{-n}  + \epsilon C_{z,n}$.
We  choose  $\epsilon =\epsilon(n)$ small enough so that $r_{n} \leq 2C(\Re(z)+ \lambda)^{-n}$.  By \eqref{eq:nussbaum} we have shown that  the essential spectral radius is not greater than $(\Re(z)+\lambda)^{-1}$. Having proved this estimate on the essential spectral radius we note that the spectral radius cannot be greater than $\Re(z)^{-1}$ otherwise there would be a contradiction with Lemma~\ref{lem:boundRz1}.
\end{proof}



\end{document}